\documentclass[11pt]{amsart}
\usepackage{amssymb}
\theoremstyle{plain}
\newtheorem{Thm}{Theorem}
\newtheorem{Lem}{Lemma}

\newtheorem{Cor}{Corollary}

\theoremstyle{remark}

\newcommand{\N}{{\mathbb N}}
\newcommand{\R}{{\mathbb R}}
\newcommand{\WW}{{\mathbb W}}
\newfont{\ccc}{eusm10 scaled\magstep1}
\newcommand{\W}{\hbox{\ccc W}}
\newcommand{\V}{\hbox{\ccc V}}
\newcommand{\eq}[1]{\eqref{E#1}}
\newcommand{\Eq}[2]{\begin{equation}\label{E#1}#2\end{equation}}

\def\and{\qquad\hbox{and}\qquad}

\def\for{\qquad\hbox{for}\quad}
\def\D#1{D^{#1}([a,b])}
\def\Wone#1;#2.{{\W\Bigl(\begin{array}{c}
  #1 \\ #2 \end{array}\Bigr)}}
\def\Wthree#1,#2,#3;#4,#5,#6.{{\W\Bigl(\begin{array}{ccc}
  #1, & #2, & #3 \\ #4, & #5, & #6 \end{array}\Bigr)}}
\def\Wfour#1,#2,#3,#4;#5,#6,#7,#8.{{\W\Bigl(\begin{array}{cccc}
  #1, & #2, & #3, & #4 \\ #5, & #6, & #7, & #8 \end{array}\Bigr)}}
\def\Wsix#1,#2,#3,#4;#5,#6,#7,#8.{{\W\Bigl(\begin{array}{cccccc}
  #1, & \dots, & #2, & #3, & \dots, & #4 \\
  #5, & \dots, & #6, & #7, & \dots, & #8 \end{array}\Bigr)}}
\def\Wseven#1,#2,#3,#4,#5;#6,#7,#8,#9.{{\W\Bigl(\begin{array}{ccccccc}
  #1, & \dots, & #2, & #3, & \dots, & #4, & #5 \\
  #6, & \dots, & #7, & #8, & \dots, & #9, & #9\end{array}\Bigr)}}

\begin{document}
\large

\date{\today}

\title{A General Mean Value Theorem}

\author{Zsolt P\'ales}
\address{Institute of Mathematics,
University of Debrecen, H-4010 Debrecen, Pf.\ 12, Hungary}
\email{pales@science.unideb.hu}

\subjclass{Primary 26A51, 26B25}
\keywords{Mean value theorem, Wronski determinant, Quasi differential operators}

\thanks{This research has been supported by the Hungarian
National Research Science Foundation (OTKA) Grant K111651.}

\begin{abstract}
In this note  a general a Cauchy-type mean value theorem for the ratio
of functional determinants is offered. It generalizes Cauchy's and
Taylor's mean value theorems as well as other classical mean value theorems.
\end{abstract}

\maketitle

\section{Introduction}

The aim of the present note is to offer a unified approach to most
of the mean value theorems known in elementary analysis.

Let $x_1,\dots,x_k$ be arbitrary points in the real interval $[a,b]$.
Then, one can uniquely determine a permutation $\pi$ of
the set $\{1,\dots,k\}$, $n\in\N$, $\xi_1<\cdots<\xi_n$ in $[a,b]$
and $k_1,\dots,k_n$ in $\N$ with $k_1+\cdots+k_n=k$ such that
\Eq{0}{
  (x_{\pi(1)},\dots,x_{\pi(k)})=
  (\underbrace{\xi_1,\dots,\xi_1}_{k_1\hbox{\tiny\ times}},
      \quad\dots,\quad
   \underbrace{\xi_n,\dots,\xi_n}_{k_n\hbox{\tiny\ times}}).
}
If $w_1,\dots,w_{m+k}:[a,b]\to\R$ is a system of $(k-1)$ times
differentiable functions $(m>0)$, and $u_1,\dots,u_{m+k}\in\R^m$
are vectors, then we define
\begin{eqnarray*}
  && \hspace{-4mm}
     \Wthree w_1,\dots,w_{m+k};u_1,\dots,u_{m+k}.(x_1,\dots,x_k)\\
  && \hspace{4mm}
   :=\left|\begin{array}{cccccccccc}
    u_{1,1}           & \dots             & u_{1,m} &
    w_1(\xi_1)        & \!\!\!\dots\!\!\! & w_1^{(k_1-1)}(\xi_1)
                      & \dots             &
    w_1(\xi_n)        & \!\!\!\dots\!\!\! & w_1^{(k_n-1)}(\xi_n) \\
    \vdots & & \vdots & \vdots & & \vdots & & \vdots & & \vdots \\
    u_{m+k,1}         & \dots             & u_{m+k,m} &
    w_{m+k}(\xi_1)    & \!\!\!\dots\!\!\! & w_{m+k}^{(k_1-1)}(\xi_1)
                      & \dots             &
    w_{m+k}(\xi_n)    & \!\!\!\dots\!\!\! & w_{m+k}^{(k_n-1)}(\xi_n)
   \end{array}\right|,
\end{eqnarray*}
where the right hand side of this equation is an $(m+k)\times (m+k)$
determinant, $w^{(i)}$ stands for the $i$th derivative of the function
$w$, $u_{i,j}$ denotes the $j$th coordinate of the vector $u_i$, and
$\xi_i$, $k_i$ is determined by \eq{0}.

We also allow $m$ to take the value $0$, with the following
notational conventions: $\R^0:=\{0\}$ and
\begin{eqnarray*}
  \Wthree w_1,\dots,w_k;u_1,\dots,u_k.(x_1,\dots,x_k)
  &:=& \W(w_1,\dots,w_k)(x_1,\dots,x_k)\\
  &:=& \left|\begin{array}{ccccccc}
        w_1(\xi_1)        & \!\!\!\dots\!\!\! & w_1^{(k_1-1)}(\xi_1)
                          & \dots &
        w_1(\xi_n)        & \!\!\!\dots\!\!\! & w_1^{(k_n-1)}(\xi_n) \\
        \vdots & & \vdots & & \vdots & & \vdots \\
        w_k(\xi_1)        & \!\!\!\dots\!\!\! & w_k^{(k_1-1)}(\xi_1)
                          & \dots &
        w_k(\xi_n)        & \!\!\!\dots\!\!\! & w_k^{(k_n-1)}(\xi_n)
       \end{array}\right|,
\end{eqnarray*}

Observe that if here $x_1=\cdots=x_k=\xi$, then the above definition
reduces to
$$
  \W(w_1,\dots,w_k)(\xi,\dots,\xi)
   =\left|\begin{array}{ccc}
        w_1(\xi)        & \cdots & w_1^{(k-1)}(\xi) \\
        \vdots          & \ddots & \vdots   \\
        w_k(\xi)  & \cdots & w_k^{(k-1)}(\xi)
    \end{array}\right|,
$$
which is known as the Wronski determinant of the system $w_1,\dots,w_k$.

The class of functions $f:[a,b]\to\R$ that are $k-1$ times continuously
differentiable on $[a,b]$ and $k$ times differentiable on the
open interval $]a,b[$ will be denoted by $D^k([a,b])$.

\medskip
Now we are able to formulate the main result of this paper.

\begin{Thm}
Let $1\le k$, $0\le m$ be integers and $u_1,\dots,u_{m+k}\in\R^m$
such that (if $0<m$ then) $u_1,\dots,u_m$ are linearly independent,
i.e.,
\Eq{1A}{
   \V_0:=
   \left|\begin{array}{lcl}
        u_{1,1}       & \cdots & u_{1,m} \\
        \ \ \vdots    & \ddots & \ \ \vdots   \\
        u_{m,1}     & \cdots & u_{m,m}
    \end{array}\right| \not=0.
}
In addition, let $w_1,\dots,w_{m+k}\in D^k([a,b])$ be a system of
functions satisfying
\Eq{1}{
  \V_n(\xi):=
  \Wthree w_1,\dots,w_{m+n};u_1,\dots,u_{m+n}.
   (\underbrace{\xi,\dots,\xi}_{n\hbox{\tiny\rm\ times}})\not=0
}
for all $\xi\in [a,b]$ and $n=1,\dots,k$. Then, for all nonidentical
points $x_1,\dots,x_{k+1}\in [a,b]$, vectors $p,q\in\R^m$ and functions
$f,g\in D^k([a,b])$, there exists an intermediate point $\xi\in]\min
x_i,\max x_i[$ such that
\Eq{2}{
  \Wfour w_1,\dots,w_{m+k},f;u_1,\dots,u_{m+k},p.(\xi,\dots,\xi)
      \cdot
  \Wfour w_1,\dots,w_{m+k},g;u_1,\dots,u_{m+k},q.(x_1,\dots,x_{k+1})
  \qquad\qquad
}
$$
  \qquad\qquad
  = \Wfour w_1,\dots,w_{m+k},g;u_1,\dots,u_{m+k},q.(\xi,\dots,\xi)
      \cdot
    \Wfour w_1,\dots,w_{m+k},f;u_1,\dots,u_{m+k},p.(x_1,\dots,x_{k+1}).
$$
\end{Thm}

The proof of this theorem is given in the next section. Now we list
some of its consequences.

\begin{Cor} {\rm(\sl Cauchy's Mean Value Theorem.\rm)} Let $f,g\in
D^1[a,b]$ Then there exists $\xi\in]a,b[$ such that
$$
  f'(\xi)(g(a)-g(b))=g'(\xi)(f(a)-f(b)).
$$
\end{Cor}

\begin{proof} Let $k=1$, $m=0$, $w_1\equiv 1$ and $x_1=a$, $x_2=b$ in
Theorem 1. Then the statement follows immediately from \eq2.
\end{proof}

\begin{Cor} {\rm(\sl Taylor's Mean Value Theorem.\rm)} Let $f\in
D^k([a,b])$. Then, for all $x\in]a,b]$, there exists $\xi\in]a,b[$
such that
$$
  f(x)=f(a)+f'(a)(x-a)+\cdots
       +\frac{f^{(k-1)}(a)}{(k-1)!}(x-a)^{k-1}+
        \frac{f^{(k)}(\xi)}{k!}(x-a)^{k}.
$$
\end{Cor}

\begin{proof} Let $m=0$,
$$
  w_i(x)=\frac{(x-a)^{i-1}}{(i-1)!} \for i=1,\dots,k \and
   g(x)=\frac{(x-a)^k}{k!}.
$$
Then, for all $\xi\in[a,b]$,
$$
  \W(w_1)(\xi)=\cdots=\W(w_1,\dots,w_k)(\xi,\dots,\xi)=1,
$$
therefore, \eq{1A} and \eq{1} are satisfied. Thus, taking
$x_1=\cdots=x_k=a$ and $x_{k+1}=x$ in Theorem 1, we obtain that
there exists $\xi\in]a,x[$ satisfying
\Eq{3}{
  \W(w_1,\dots,w_k,f)(\xi,\dots,\xi,\xi)
        \cdot \W(w_1,\dots,w_k,g)(a,\dots,a,x)
  \qquad\qquad\qquad
}
$$
  \qquad\qquad\qquad
  =\W(w_1,\dots,w_k,g)(\xi,\dots,\xi,\xi)
         \cdot \W(w_1,\dots,w_k,f)(a,\dots,a,x).
$$
A simple computation yields that
$$
  \W(w_1,\dots,w_k,f)(a,\dots,a,x)=
  f(x)-f(a)-f'(a)(x-a)-\cdots-\frac{f^{(k-1)}(a)}{(k-1)!}(x-a)^{k-1},
$$
$$
  \W(w_1,\dots,w_k,g)(a,\dots,a,x)=\frac{(x-a)^k}{k!},
$$
furthermore
$$
  \W(w_1,\dots,w_k,f)(\xi,\dots,\xi,\xi)=f^{(k)}(\xi)\and
  \W(w_1,\dots,w_k,g)(\xi,\dots,\xi,\xi)=1.
$$
Thus, Taylor's theorem follows from \eq{3} at once.
\end{proof}

Let $w_1(x)=1,\dots,w_k(x)=x^{k-1},w_{k+1}=x^k$ for $x\in[a,b]$ and
let $a\le x_1\le \cdots\le x_{k+1}\le b$ with $a<x_{k+1}$ and $x_1<b$.
Then the ratio
$$
  [x_1,\dots,x_{k+1}]f:=
    \frac{\W(w_1,\dots,w_k,f)(x_1,\dots,x_k,x_{k+1})}
     {\W(w_1,\dots,w_k,w_{k+1})(x_1,\dots,x_k,x_{k+1})}
$$
is called the $k$th order divided difference of $f\in D^k([a,b])$
over the points $x_1,\dots,x_{k+1}$ (c.f. \cite[p. 45]{SCH}). Divided
differences are usually defined in an inductive way in the literature,
see e.g. \cite[Sect. 3.17]{AH} and \cite[Sect. 2.3]{HA}. The proof
of the equivalence of the above definition to the usual one can be
found in \cite[Theorem 2.51, p.47]{SCH}.

Concerning divided differences, the following result is well known
(c.f.\ \cite[p. 274]{AH} and \cite[(2.93)]{SCH}).

\begin{Cor} Let $f\in\D k$ and $a\le x_1\le \cdots\le x_{k+1}\le b$
with $x_1<x_{k+1}$. Then there exists $\xi\in]x_1,x_{k+1}[$ such that
\Eq{3A}{
  [x_1,\dots,x_{k+1}]f=\frac{f^{(k)}(\xi)}{k!}.
}
\end{Cor}

\begin{proof} Apply Theorem 1 when $m=0$ with the function
$g(x)=w_{k+1}(x)=x^k$.
Then we find that there exists $\xi\in]x_1,x_{k+1}[$ such that
$$
  [x_1,\dots,x_{k+1}]f
     =[\underbrace{\xi,\dots,\xi}_{k+1\hbox{\tiny\ times}}]f
     =\frac{f^{(k)}(\xi)}{k!}.
$$
Thus \eq{3A} is proved. \end{proof}

The following result, called Cauchy Mean Value Theorem, is due to
R\"atz and Russel \cite{RR}.

\begin{Cor} Let $f,g\in\D k$ such that $g^{(k)}(\xi)\not=0$ for
$\xi\in]a,b[$ and let $a\le x_1\le \cdots\le x_{k+1}\le b$ with
$x_1<x_{k+1}$. Then there exists $\xi\in]x_1,x_{k+1}[$ such that
\Eq{3B}{
  \frac{[x_1,\dots,x_{k+1}]f}{[x_1,\dots,x_{k+1}]g}
    =\frac{f^{(k)}(\xi)}{g^{(k)}(\xi)}.
}
\end{Cor}

\begin{proof} Applying Corollary 3 for the function $g$ first,
we can observe that
$$
   [x_1,\dots,x_{k+1}]g\not=0.
$$
Hence the left hand side of \eq{3B} exists. Clearly,
$$
  \frac{[x_1,\dots,x_{k+1}]f}{[x_1,\dots,x_{k+1}]g}
    = \frac{\W(w_1,\dots,w_k,f)(x_1,\dots,x_k,x_{k+1})}
      {\W(w_1,\dots,w_k,g)(x_1,\dots,x_k,x_{k+1})}
$$
Therefore, by Theorem 1, there exists $\xi\in]x_1,x_{k+1}[$ such that
$$
  \frac{[x_1,\dots,x_{k+1}]f }{ [x_1,\dots,x_{k+1}]g}
    = \frac{\W(w_1,\dots,w_k,f)(\xi,\dots,\xi,\xi)}
       {\W(w_1,\dots,w_k,g)(\xi,\dots,\xi,\xi)}
    =\frac{f^{(k)}(\xi)}{g^{(k)}(\xi)},
$$
whence \eq{3B} follows.
\end{proof}

\section{Proof of the main result}

In the proof of Theorem 1, we shall need the following notion:
A function $f\in D^k([a,b])$ {\em vanishes $k+1$ times in} $[a,b]$ if
there exist points $x_1<\cdots<x_n$ in $[a,b]$ with $x_1<b$, $a<x_n$
and natural numbers $k_1,\dots,k_n$ with $k_1+\cdots+k_n=k+1$ such that
\Eq{4}{
  f^{(j)}(x_i)=0 \for j=0,\dots,k_i-1,\,\,i=1,\dots,n.
}
For instance, the function $f(x)=x(x-1)$ vanishes twice in $[0,1]$.
However the function $f(x)=x^2$ does not vanish twice in $[0,1]$, but
it does in $[-1,1]$, (that is, all the zeroes of $f$ should not be
concentrated at the endpoints of the interval).

We recall the following lemmas from \cite{P} and, for readers
convenience, we also provide their proofs.

\begin{Lem}
If $f,g\in D^k([a,b])$ and $f$ vanishes $k+1$ times in $[a,b]$,
then $fg$ also vanishes $k+1$ times in $[a,b]$.
\end{Lem}

\begin{proof}
By the assumption, there are $x_1<\cdots<x_n$ in $[a,b]$ with $x_1<b$,
$a<x_n$ and $k_1,\dots,k_n\in\N$ with $k_1+\cdots+k_n=k+1$ such that
\eq{4} holds. Then, using Leibniz's Product Rule, one can check that
$$
  (fg)^{(j)}(x_i)=0 \for j=0,\dots,k_i-1,\,\,i=1,\dots,n.
$$
Thus $fg$ also vanishes $k+1$ times in $[a,b]$.
\end{proof}

\begin{Lem}
If $f\in D^k([a,b])$ vanishes $k+1$ times in $[a,b]$, then $f'$
vanishes $k$ times in $[a,b]$.
\end{Lem}

\begin{proof}
We have \eq{4} for some $x_1<\cdots<x_n$ with
$x_1<b$, $a<x_n$ and $k_1,\dots,k_n\in\N$ with $k_1+\cdots+k_n=k+1$.
If $n=1$, then there is nothing to prove. Otherwise, by Rolle's
Mean Value Theorem, there exist $x_i<\xi_i<x_{i+1}$ such that
$$
  f'(\xi_i)=0 \for i=1,\dots,n-1.
$$
These equalities combined with \eq{4} yield that $f'$ vanishes $k$
times on $[a,b]$.
\end{proof}

The following lemma generalizes \cite[Lemma 3]{P}. The result
obtained therein corresponds the case $m=0$ below.

\begin{Lem}
Let $1\le k$, $0\le m$ be integers and $u_1,\dots,u_{m+k}\in\R^m$
such that \eq{1A} holds (if $m>0$). Let $w_1,\dots,w_m\in D^k([a,b])$
be a system of functions satisfying \eq{1} for all $\xi\in[a,b]$.
For $f\in D^k([a,b])$, define the following operators
$$
  \WW_n(f)(\xi)
   :=\Wfour w_1,\dots,w_{m+n},f;u_1,\dots,u_{m+n},0.
     (\underbrace{\xi,\dots,\xi}_{n+1\hbox{\tiny\rm\ times}}),
     \qquad n=0,\dots,k.
$$
where $\xi\in[a,b]$ if $n<k$, and $\xi\in]a,b[$ if $n=k$.
Then the following recursive formula
\Eq{5}{
  \WW_n(f)(\xi)
  = \frac{d}{d\xi} \left(\frac{\WW_{n-1}(f)(\xi)}{\V_n(\xi)}\right)
    \cdot \frac{\left[\V_n(\xi)\right]^2}{\V_{n-1}(\xi)}
}
holds for all $\xi\in[a,b]$ if $1\le n<k$, and for all $\xi\in]a,b[$ if
$n=k$. (Here $\V_0$ and $\V_n$ ($1\le n\le k$) are defined in \eq{1A}
and in \eq{1}, respectively. In the case $m=0$ we set $\V_0=0$.)
\end{Lem}

\begin{proof} The argument described below works for $m\not=0$. The
$m=0$ case is completely analogous, therefore omitted.

The vectors $u_1,\dots,u_m$ are linearly independent
in $\R^m$, hence they form a basis of $\R^m$. Thus, we can find real
numbers $\gamma_{1,n},\dots,\gamma_{m,n}$ such that, for $n=1,\dots,k$,
\Eq{5A}{
  -u_{m+n}=\gamma_{1,n}u_1+\cdots+\gamma_{m,n}u_m.
}
Then define the functions $v_n:[a,b]\to\R$ by
\Eq{5B}{
  v_n:= w_{m+n}+\gamma_{1,n}w_1+\cdots+\gamma_{m,n}w_m.
}
Now we show that the functions $v_1,\dots,v_n$ form a linearly
independent system of solutions of the equation
\Eq{6}{
  \WW_n(f)(\xi)=0,\qquad \xi\in]a,b[
}
which is an $n$th order homogeneous linear differential equation for
the unknown function $f$.

To see this, we first compute $\WW_n(v_j)$ for any $1\le j\le k$ and
$0\le n\le k$. Multiplying the  first $m$ rows of the determinant
$\WW_n(v_j)$ by $\gamma_{1,j},\dots,\gamma_{m,j}$, respectively,
subtracting their sum from the last row, then using \eq{5A}, we get
\begin{eqnarray*}
  \WW_n(v_j)(\xi)
   &=&\Wfour w_1,\dots,w_{m+n},w_{m+j}+\sum_{i=1}^m\gamma_{i,j}w_i;%
      u_1,\dots,u_{m+n},0.(\xi,\dots,\xi)\\
   &=&\Wfour w_1,\dots,w_{m+n},w_{m+j};%
      u_1,\dots,u_{m+n},-\sum_{i=1}^m\gamma_{i,j}u_i.(\xi,\dots,\xi)\\
   &=&\Wfour w_1,\dots,w_{m+n},w_{m+j};%
      u_1,\dots,u_{m+n},u_{m+j}.(\xi,\dots,\xi)=0
\end{eqnarray*}
If $j\le n$, then this formula results $\WW_n(v_j)=0$. On the other
hand, with $j=n+1$, we have $\WW_n(v_{n+1})=\V_{n+1}$.

The function $v_1$ cannot be identically zero because
$\WW_0(v_1)=\V_1\not=0$. Hence $\{v_1\}$ is a linearly independent set
of solutions of $\WW_1(f)=0$. Assume now that $v_1,\dots,v_n$ form a
linearly independent system of solutions of $\WW_n(f)=0$. The function
$v_{n+1}$ is not a solution of this equation, hence, it cannot be a
linear combination of $v_1,\dots,v_n$. Thus, $v_1,\dots,v_{n+1}$ is also
a linearly independent system.

Temporarily, denote the operator defined by the right hand side of
\eq{5} by $\WW^*_n(f)$. It is clear that $\WW^*_n(f)$ is also an
$n$th-order linear differential operator of $f$. We show that
$v_1,\dots,v_n$ also solves the equation $\WW^*_n(f)=0$. This is
obvious if $f=v_1,\dots,v_{n-1}$ (since these functions are solutions
of the equation $\WW_{n-1}(f)=0$). On the other hand
$$
  \WW^*_n(v_n)(\xi)
  = \frac{d}{d\xi} \left(\frac{\WW_{n-1}(v_n)(\xi)}{\V_n(\xi)}\right)
    \cdot \frac{\left[\V_n(\xi)\right]^2}{\V_{n-1}(\xi)}
  = \frac{d}{d\xi} \left(\frac{\V_n(\xi)}{\V_n(\xi)}\right)
    \cdot \frac{\left[\V_n(\xi)\right]^2}{\V_{n-1}(\xi)}=0.
$$

Observe that the coefficients of $f^{(n)}$ in $\WW_n(f)$ and
$\WW^*_n(f)$ are equal to $\V_n$ which does not vanish anywhere
in $[a,b]$. Therefore, having the same solution space, these two
operators have to coincide for all $1\le n\le k$.
\end{proof}

\begin{Lem}
Let $1\le k$, $0\le m$ be integers and $u_1,\dots,u_{m+k}\in\R^m$
such that \eq{1A} holds (if $m>0$). Let $w_1,\dots,w_m\in D^k([a,b])$
be a system of functions satisfying \eq{1} for all $\xi\in[a,b]$.
Assume that the function $f\in D^k([a,b])$ vanishes
$k+1$ times in $[a,b]$. Then, for each $0\le n\le k$, the function
$\WW_n(f)$ defined in Lemma 3 vanishes $k+1-n$ times in $[a,b]$.
\end{Lem}

\begin{proof}
We prove by induction. If $n=0$, then $\WW_0(f)=\V_0f$, hence, in
this case, there is nothing to prove. Let $n\ge1$ and assume that
$\WW_{n-1}(f)$ vanishes $k+1-(n-1)$ times. Then, applying Lemma 1
and Lemma 2, one sees that the function
$$
  \frac{d}{d\xi} \left(\frac{\WW_{n-1}(f)(\xi)}{\V_n(\xi)}\right)
    \cdot \frac{\left[\V_n(\xi)\right]^2}{\V_{n-1}(\xi)}
  \qquad (\xi\in[a,b])
$$
vanishes $k+1-(n-1)-1=k+1-n$ times. Therefore, due to the recursive
formula established in Lemma 3, $\WW_{n}(f)$ vanishes $k+1-n$ times.
\end{proof}

Now we are ready to prove the main result of the paper.

\begin{proof}[Proof of Theorem 1] Let $x_1\le\cdots\le x_{k+1}$
be in $[a,b]$ with $\min x_i<\max x_i$. Then there exist a
permutation $\pi$ of $\{1,\dots,x_{k+1}\}$, $n\in\N$,
$\xi_1<\cdots<\xi_n$ in $[a,b]$ and $k_1,\dots,k_n\in\N$ with
$k_1+\cdots+k_n=k+1$ such that
\Eq{9}{
  (x_{\pi(1)},\dots,x_{\pi(k+1)})=
  (\underbrace{\xi_1,\dots,\xi_1}_{k_1\hbox{\tiny\ times}},
    \quad\dots,\quad
   \underbrace{\xi_n,\dots,\xi_n}_{k_n\hbox{\tiny\ times}})
}
holds. Define the function $F:[a,b]\to\R$ by
$$
  F(x)
   :=\left|\begin{array}{ccccccccc}
        u_{1,1}        & \!\!\!\dots\!\!\! & u_{1,m}   &
        w_1(\xi_1)     & \!\!\!\dots\!\!\! & w_1^{(k_1-1)}(\xi_1)
                       & \dots &            w_1^{(k_n-1)}(\xi_n)
                       & w_1(x) \\
        \vdots         & & \vdots   &
        \vdots         & & \vdots & & \vdots & \vdots \\
        u_{m+k,1}      & \!\!\!\dots\!\!\! & u_{m+k,m} &
        w_{m+k}(\xi_1) & \!\!\!\dots\!\!\! & w_{m+k}^{(k_1-1)}(\xi_1)
                       & \dots &            w_{m+k}^{(k_n-1)}(\xi_n)
                       & w_{m+k}(x) \\
        p_{1}          & \!\!\!\dots\!\!\! & p_{m} &
        f(\xi_1)       & \!\!\!\dots\!\!\! & f^{(k_1-1)}(\xi_1)
                       & \dots &             f^{(k_n-1)}(\xi_n)
                       & f(x) \\
        q_{1}          & \!\!\!\dots\!\!\! & q_{m} &
        g(\xi_1)       & \!\!\!\dots\!\!\! & g^{(k_1-1)}(\xi_1)
                       & \dots &             g^{(k_n-1)}(\xi_n)
                       & g(x)
    \end{array}\right|.
$$
It is obvious at once that
$$
  F^{(j)}(\xi_i)=0 \for j=0,\dots,k_i-1,\,\,i=1,\dots,n,
$$
therefore $F$ vanishes $k+1$ times in $[a,b]$. Thus, by Lemma 4,
there exists $\xi\in]a,b[$ such that
\Eq{10}{
   \WW_k(F)(\xi)=
   \Wfour w_1,\dots,w_{m+k},F;u_1,\dots,u_{m+k},0.
     (\underbrace{\xi,\dots,\xi}_{k+1\hbox{\tiny\rm\ times}})=0.
}
Now determine the contstants $\gamma_{i,n}$ such that they satisfy
\eq{5A} and define $v_1,\dots,v_k$ by \eq{5B}. Similarly, choose
$\alpha_1,\dots,\alpha_m$ and $\beta_1,\dots,\beta_m$ such that
\Eq{11}{
   -p=\alpha_1u_1+\cdots+\alpha_mu_m
   \qquad \hbox{and}\qquad
   -q=\beta_1u_1+\cdots+\beta_mu_m.
}
and define
\Eq{12}{
   \phi=f+\alpha_1w_1+\cdots+\alpha_mw_m
   \qquad \hbox{and}\qquad
   \psi=g+\beta_1w_1+\cdots+\beta_mw_m.
}
Using these choices of the constants, add linear combination of the
first $m$ rows of $F$ to the rest of the rows to obtain
\begin{eqnarray*}
  F(x)
   &:=&\left|\begin{array}{ccccccccc}
        u_{1,1}        & \!\!\!\dots\!\!\! & u_{1,m}   &
        w_1(\xi_1)     & \!\!\!\dots\!\!\! & w_1^{(k_1-1)}(\xi_1)
                       & \dots &             w_1^{(k_n-1)}(\xi_n)
                       & w_1(x) \\
        \vdots         & & \vdots   &
        \vdots         & & \vdots & & \vdots & \vdots \\
        u_{m,1}        & \!\!\!\dots\!\!\! & u_{m,m}   &
        w_m(\xi_1)     & \!\!\!\dots\!\!\! & w_m^{(k_1-1)}(\xi_1)
                       & \dots &             w_m^{(k_n-1)}(\xi_n)
                       & w_m(x) \\
        0              & \!\!\!\dots\!\!\! & 0         &
        v_{1}(\xi_1)   & \!\!\!\dots\!\!\! & v_{1}^{(k_1-1)}(\xi_1)
                       & \dots &             v_{1}^{(k_n-1)}(\xi_n)
                       & v_{1}(x) \\
        \vdots         & & \vdots   &
        \vdots         & & \vdots & & \vdots & \vdots \\
        0              & \!\!\!\dots\!\!\! & 0         &
        v_{k}(\xi_1)   & \!\!\!\dots\!\!\! & v_{k}^{(k_1-1)}(\xi_1)
                       & \dots &             v_{k}^{(k_n-1)}(\xi_n)
                       & v_{k}(x) \\
        0              & \!\!\!\dots\!\!\! & 0     &
        \phi(\xi_1)    & \!\!\!\dots\!\!\! & \phi^{(k_1-1)}(\xi_1)
                       & \dots &             \phi^{(k_n-1)}(\xi_n)
                       & \phi(x) \\
        0              & \!\!\!\dots\!\!\! & 0     &
        \psi(\xi_1)    & \!\!\!\dots\!\!\! & \psi^{(k_1-1)}(\xi_1)
                       & \dots &             \psi^{(k_n-1)}(\xi_n)
                       & \psi(x)
    \end{array}\right| \\
   &=& \V_0\cdot\left|\begin{array}{cccccc}
        v_{1}(\xi_1)   & \!\!\!\dots\!\!\! & v_{1}^{(k_1-1)}(\xi_1)
                       & \dots &             v_{1}^{(k_n-1)}(\xi_n)
                       & v_{1}(x) \\
        \vdots         & & \vdots & & \vdots & \vdots \\
        v_{k}(\xi_1)   & \!\!\!\dots\!\!\! & v_{k}^{(k_1-1)}(\xi_1)
                       & \dots &             v_{k}^{(k_n-1)}(\xi_n)
                       & v_{k}(x) \\
        \phi(\xi_1)    & \!\!\!\dots\!\!\! & \phi^{(k_1-1)}(\xi_1)
                       & \dots &             \phi^{(k_n-1)}(\xi_n)
                       & \phi(x) \\
        \psi(\xi_1)    & \!\!\!\dots\!\!\! & \psi^{(k_1-1)}(\xi_1)
                       & \dots &             \psi^{(k_n-1)}(\xi_n)
                       & \psi(x)
    \end{array}\right|.
\end{eqnarray*}
Expanding by the last column, we get
$$
  F(x)=\sum_{i=1}^k C_i\cdot v_i(x) - A \phi(x)+ B \psi(x),
$$
where $A,B,C_i$ are the values of the corresponding subdeterminants.
Substituting the above form of $F$ into \eq{10}, and using that
$\WW_k(v_i)=0$, we get that
\Eq{13}{
        A \cdot\WW_k(\phi)(\xi) = B \cdot\WW_k(\psi)(\xi)
}
In the rest of the proof we show that \eq{13} reduces to \eq{2}.

Indeed, adding a certain linear combination to the last row of
$\WW_k$, we get
\begin{eqnarray*}
  \WW_k(\phi)(\xi)
   &=& \Wfour w_1,\dots,w_{m+k},f+\sum_{i=1}^m\alpha_iw_i;%
        u_1,\dots,u_{m+k},0.(\xi,\dots,\xi) \\
   &=& \Wfour w_1,\dots,w_{m+k},f;%
        u_1,\dots,u_{m+k},-\sum_{i=1}^m\alpha_iu_i.(\xi,\dots,\xi)\\
   &=& \Wfour w_1,\dots,w_{m+k},f;%
        u_1,\dots,u_{m+k},p.(\xi,\dots,\xi).
\end{eqnarray*}
For the constant $A$, due to its origin, we have
$$
  A=\V_0\cdot\left|\begin{array}{ccccc}
        v_{1}(\xi_1)   & \!\!\!\dots\!\!\! & v_{1}^{(k_1-1)}(\xi_1)
                       & \dots &             v_{1}^{(k_n-1)}(\xi_n)\\
        \vdots         & & \vdots & & \vdots \\
        v_{k}(\xi_1)   & \!\!\!\dots\!\!\! & v_{k}^{(k_1-1)}(\xi_1)
                       & \dots &             v_{k}^{(k_n-1)}(\xi_n)\\
        \psi(\xi_1)    & \!\!\!\dots\!\!\! & \psi^{(k_1-1)}(\xi_1)
                       & \dots &             \psi^{(k_n-1)}(\xi_n)
    \end{array}\right|.
$$
Now, using an argument similar to that applied in the computation of
$F$, one can get that
\begin{eqnarray*}
  A &=&\left|\begin{array}{cccccccc}
        u_{1,1}        & \!\!\!\dots\!\!\! & u_{1,m}   &
        w_1(\xi_1)     & \!\!\!\dots\!\!\! & w_1^{(k_1-1)}(\xi_1)
                       & \dots &             w_1^{(k_n-1)}(\xi_n)\\
        \vdots         & & \vdots   &
        \vdots         & & \vdots & & \vdots \\
        u_{m,1}        & \!\!\!\dots\!\!\! & u_{m,m}   &
        w_m(\xi_1)     & \!\!\!\dots\!\!\! & w_m^{(k_1-1)}(\xi_1)
                       & \dots &             w_m^{(k_n-1)}(\xi_n)\\
        0              & \!\!\!\dots\!\!\! & 0         &
        v_{1}(\xi_1)   & \!\!\!\dots\!\!\! & v_{1}^{(k_1-1)}(\xi_1)
                       & \dots &             v_{1}^{(k_n-1)}(\xi_n)\\
        \vdots         & & \vdots   &
        \vdots         & & \vdots & & \vdots \\
        0              & \!\!\!\dots\!\!\! & 0         &
        v_{k}(\xi_1)   & \!\!\!\dots\!\!\! & v_{k}^{(k_1-1)}(\xi_1)
                       & \dots &             v_{k}^{(k_n-1)}(\xi_n)\\
        0              & \!\!\!\dots\!\!\! & 0     &
        \psi(\xi_1)    & \!\!\!\dots\!\!\! & \psi^{(k_1-1)}(\xi_1)
                       & \dots &             \psi^{(k_n-1)}(\xi_n)
      \end{array}\right| \\
    &=&\left|\begin{array}{cccccccc}
        u_{1,1}        & \!\!\!\dots\!\!\! & u_{1,m}   &
        w_1(\xi_1)     & \!\!\!\dots\!\!\! & w_1^{(k_1-1)}(\xi_1)
                       & \dots &            w_1^{(k_n-1)}(\xi_n)\\
        \vdots         & & \vdots   &
        \vdots         & & \vdots & & \vdots \\
        u_{m+k,1}      & \!\!\!\dots\!\!\! & u_{m+k,m} &
        w_{m+k}(\xi_1) & \!\!\!\dots\!\!\! & w_{m+k}^{(k_1-1)}(\xi_1)
                       & \dots &            w_{m+k}^{(k_n-1)}(\xi_n)\\
        q_{1}          & \!\!\!\dots\!\!\! & q_{m} &
        g(\xi_1)       & \!\!\!\dots\!\!\! & g^{(k_1-1)}(\xi_1)
                       & \dots &             g^{(k_n-1)}(\xi_n)
      \end{array}\right|.\\
    &=& \Wfour w_1,\dots,w_{m+k},g;%
        u_1,\dots,u_{m+k},q.(x_1,\dots,x_{k+1}).
\end{eqnarray*}
Thus, we have checked that the left hand side of \eq{13} coincides with
that of \eq{2}. The equality of the right hand sides follows similarly,
and therefore, the proof is complete.
\end{proof}

We now derive a useful consequence of Theorem 1.

\begin{Thm}
Let $I\subset\R$ be an interval and $[a,b]$ be a proper subinterval
of$I$. Let $1\le k$, $1\le m$ be integers and $y_1,\dots,y_m\in
I\setminus[a,b]$. Assume that $w_1,\dots,w_{m+k}:I\to\R$ are
sufficently many times differentiable functions such that
\Eq{14}{
  \W(w_1,\dots,w_{m+n})(y_1,\dots,y_m,
   \underbrace{\xi,\dots,\xi}_{n\hbox{\tiny\rm\ times}})\not=0
}
for all $\xi\in [a,b]$ and $n=0,\dots,k$. Then, for all
nonidentical points $x_1,\dots,x_{k+1}\in [a,b]$, and functions
sufficently many times differentiable$f,g:I\to\R$, there exists an
intermediate point $\xi\in]\min x_i,\max x_i[$ such that
\Eq{15}{
  \W(w_1,\dots,w_{m+k},f)(y_1,\dots,y_m,
   \underbrace{\xi,\dots,\xi}_{k+1\hbox{\tiny\rm\ times}})
  \cdot \W(w_1,\dots,w_{m+k},g)(y_1,\dots,y_m,x_1,\dots,x_{k+1})
  \quad
}
$$
  \qquad\quad
  =\W(w_1,\dots,w_{m+k},g)(y_1,\dots,y_m,
   \underbrace{\xi,\dots,\xi}_{k+1\hbox{\tiny\rm\ times}})
  \cdot \W(w_1,\dots,w_{m+k},f)(y_1,\dots,y_m,x_1,\dots,x_{k+1}).
$$
\end{Thm}

\begin{proof}
Let $\pi$ be a permutation of the set
$\{1,\dots,m\}$, $\eta_1,\dots,\eta_l\in I$, and $m_1,\dots,m_l\in\N$
with $m_1+\cdots+m_l=m$ such that
$$
  (y_{\pi(1)},\dots,y_{\pi(m)})=
  (\underbrace{\eta_1,\dots,\eta_1}_{m_1\hbox{\tiny\ times}},
  \quad\dots,\quad
  \underbrace{\eta_l,\dots,\eta_l}_{m_l\hbox{\tiny\ times}}).
$$
Set, for $i=1,\dots,m+k$,
$$
  u_i:=(u_{i1},\dots,u_{im})
     :=(w_i(\eta_1),\dots,w_i^{(m_1-1)}(\eta_1),
     \quad\dots,\quad w_i(\eta_l),\dots,w_i^{(m_l-1)}(\eta_l)),
$$
and
$$
  p:=(p_{1},\dots,p_{m})
   :=(f(\eta_1),\dots,f^{(m_1-1)}(\eta_1),
   \quad\dots,\quad f(\eta_l),\dots,f^{(m_l-1)}(\eta_l)),
$$
$$
  q:=(q_{1},\dots,q_{m})
   :=(g(\eta_1),\dots,g^{(m_1-1)}(\eta_1),
   \quad\dots,\quad g(\eta_l),\dots,g^{(m_l-1)}(\eta_l)).
$$
Observe, that with this notations, the conditions of Theorem 1 are
satisfied and therefore there exists $\xi$ such that \eq{2} holds. It is
immediate to see that \eq{2} is equivalent to \eq{15}.
\end{proof}

\end{document}